\documentclass{amsart}       % onecolumn (second format)
\usepackage{graphicx}
\usepackage{amssymb,amsmath,amsthm}
\usepackage{color}
\usepackage{enumitem}

\renewcommand{\Re}{\mathbb{R}}

\newcommand{\conv}{\mathrm{conv}}

\newcommand{\dd}[1]{\left \langle {#1}\right \rangle}
\newcommand{\ddp}{\dd{\cdot,\cdot}}

\newcommand{\sH}{\mathcal{H}}

\renewcommand{\sp}{\mathrm{span}}
\newcommand{\sK}{\mathcal{K}}
\newcommand{\oh}{\overline{h}}
\newcommand{\hchi}{\widehat{\chi}}

\newcommand{\mat}[1]{\left[\begin{matrix} #1  \end{matrix} \right]}
\newcommand{\dpp}[1]{\langle \hspace{-0.1cm} \langle #1  \rangle \hspace{-0.1cm} \rangle}
\newcommand{\dpt}[1]{\dpp{#1}}

\newcommand{\sA}{\mathcal{A}}

\newcommand{\sfrac}[2]{{\mbox{$\frac{#1}{#2}$}}}
\newcommand{\half}{{1/2}}

\newcommand{\ssC}{\mathfrak{C}}
\newcommand{\ssS}{\mathbb{S}}

\newtheorem{theorem}{Theorem}
\newtheorem{definition}{Definition}
\newtheorem{example}{Example}
\newtheorem{corollary}[theorem]{Corollary}
\newtheorem{proposition}[theorem]{Proposition}
\newtheorem{lemma}[theorem]{Lemma}

\newcommand{\qqed}{ }

\begin{document}

\title{Inner products for Convex Bodies%\thanks{Grants or other notes
}

%\titlerunning{Minkowski inner products}        % if too long for running head

\author{David Bryant}
\address{DB: (Corresponding Author) Department of Mathematics and Statistics, University of Otago, Dunedin, New Zealand}
\email{david.bryant@otago.ac.nz}
\author{Petru Cioica-Licht}
\address{PC: Department of Mathematics and Statistics, University of Otago, Dunedin, New Zealand}
\email{pcioica@maths.otago.ac.nz}
\author{Lisa Orloff Clark}
\address{LOC: School of Mathematics and Statistics, Victoria University of Wellington, Wellington, New Zealand}
\email{lisa.clark@vuw.ac.nz}
\author{Rachael Young}
\address{RY: Department of Mathematics and Statistics, University of Otago, Dunedin, New Zealand}
\email{rachael.gray.young@gmail.com}


\begin{abstract}
We define a set inner product to be a function on pairs of convex bodies which is symmetric, Minkowski linear in each dimension, positive definite, and satisfies the natural analogue of the Cauchy-Schwartz inequality (which is not implied by the other conditions). We show that any set inner product can be embedded into an inner product space on the associated support functions, thereby extending fundamental results of H\"ormander and R\aa dstr\"om. The set inner product provides a geometry on the space of convex bodies. We explore some of the properties of that geometry, and  discuss an application of these ideas to the reconstruction of ancestral ecological niches in evolutionary biology. \\
{\bf Keywords} Inner product; convex body; Minkowski linear functionals; ecological niche\\ % \PACS{PACS code1 \and PACS code2 \and more} \\
 {\bf Math. Subject Classification}: 52A20, 52A27, 05C05.
\end{abstract}
\maketitle

\section{Introduction}
\label{intro}

A {\em convex body} is a closed, bounded, non-empty convex set. We use $\sK$ or $\sK_n$ to denote the set of convex bodies contained in $\Re^n$.  There is an extensive literature on convex bodies and their mathematics, with  metric spaces and even Banach spaces defined on $\sK$ (see \cite{Schneider2014a} for a comprehensive review). Here we investigate an  analog of a
real {\em inner product space} defined on $\sK$, motivated by applications in approximation and inference of convex bodies.

Our starting point is the definition of a {\em Minkowski linear functional} on $\sK$: a function $f:\sK \rightarrow \Re$ satisfying $f(\alpha A + \beta B) = \alpha f(A) + \beta f(B)$ for all $A,B \in \sK$ and $\alpha,\beta \geq 0$. A conventional {real} inner product is bilinear and positive definite. Replacing `bilinear' with `Minkowski bilinear' we obtain the following axioms for a set inner product $\ddp : \sK \times \sK \rightarrow \Re$.
\begin{enumerate}
\item[(A1)] Symmetry: $\dd{A,B} = \dd{B,A}$ for all $A,B \in \sK$.
\item[(A2)] Minkowski linearity with respect to the first variable: for all $\alpha,\beta \in \Re_{\geq 0}$ and $A,B,C \in \sK$,
\begin{equation}
\dd{\alpha A + \beta B , C} = \alpha \dd{A,C} + \beta \dd{B,C}. \label{eq:MinkowLinear}\end{equation}
\item[(A3)] Positive definite: $\dd{A,A} > 0$ when $A \neq  \{0\}$.
\end{enumerate}
The key difference between these axioms and those for a standard inner product is axiom (A2), as $\dd{\alpha A,B} = \alpha \dd{A,B}$ is only required for $\alpha \geq 0$.  
This small change makes a significant difference, including the fact that the  Cauchy-Schwartz Inequality  $\dd{A,B}^2 \leq \dd{A,A} \dd{B,B}$ does not follow from (A1)---(A3), as we prove in the following section.  We include the Cauchy-Schwartz inequality directly  as a fourth axiom
\begin{enumerate}
\item[(A4)] For all $A,B \in \sK$, $\dd{A,B}^2 \leq \dd{A,A} \dd{B,B}$  with equality if and only if  $A = \lambda B$ for some $\lambda \in \Re_{\geq 0}$.
\end{enumerate}

\begin{definition}
A {\bf set inner product} on $\sK$ is a function $\ddp : \sK \times \sK \rightarrow \Re$  satisfying (A1), (A2), (A3) and (A4). 
\end{definition}

%
%
%
%As an example, let $\sK$ be the set of closed, bounded intervals on the real line. Define 
%\[ \dd{A,B} = \max(A) \max(B) + \min(A) \min(B).\]
%Then (A1) and (A3) follow immediately. For $A,B \in \sK$ and $\alpha,\beta \in \Re_{\geq 0}$ we have $\max(\alpha A + \beta B) = \alpha \max(A) + \alpha \max(B)$  and  $\min(\alpha A + \beta B) = \alpha \min(A) + \alpha \min(B)$, giving (A2). Axiom (A4) is a consequence of
%\[ \dd{A,A} \dd{B,B} - \dd{A,B}^2 = (\max(A)\max(B) - \min(A)\min(B))^2  \geq 0.\]
%
%

Note that for any $a,b \in \Re^n$ we have
\[ \dd{\{a\},\{b\}} + \dd{-\{a\},\{b\}} = \dd{\{a\} - \{a\},\{b\}} = \dd{\{0\},\{b\}} = 0\]
so that $\ddp$ is bilinear when restricted to singletons. Thus the restriction of a set inner product to singleton sets is an inner product, called the {\bf induced inner product}.  
We also note that for $A \in \sK$,  $-A:=-1(A)$ is not the additive inverse of $A$ and hence $\sK$ is not a vector space.  So we can't replace (A2) by the stronger version
\begin{enumerate}
\item[(A2')\hspace{-1mm}] For all $\alpha,\beta \in \Re$ and $A,B,C \in \sK,\, 
\dd{\alpha A + \beta B , C} = \alpha \dd{A,C} + \beta \dd{B,C}
$
\end{enumerate}
because it leads to a contradiction: since $A - A = -(A-A)$and  $\dd{-B,B} = -\dd{B,B}$ for all $A,B \in \sK$ implies $\dd{A-A,A-A} = - \dd{A-A,A-A} = 0$ contradicting (A3).

\begin{example}
Let $\ddp:\sK_1 \times \sK_1 \rightarrow \Re$ be the function on closed and bounded intervals of $\Re$ given by 
\[\dd{[a,b],[c,d]} = ac + bd.\]
Then $\ddp$ is a set inner product.
\end{example}

In this paper we investigate  set inner products, with an eye to its potential application to the approximation and statistical inference of convex sets. Our main result is that set inner products are exactly inner products for the set of support functions (Theorem~\ref{thm:equiv}). This observation therefore extends the classical results of  \cite{Hormander1955a,Radstrom1952a} on embedding $\sK$ in a Banach space.

On one hand Theorem~\ref{thm:equiv} is a negative result: it shows that set inner products will not lead to fundamentally new mathematics which is not simply a consequence of standard functional analysis. On the other hand we obtain access to vast functional analysis toolkit to apply to the geometry and statistics of convex sets. We conclude the  second section by demonstrating that axiom (A4), the Cauchy-Schwartz inequality, is not implied by axioms (A1)---(A3).  

The geometry of $\sK$ is the subject of the third section. We examine lines and subspaces in the set inner product space, and establish basic orthogonality results. We make extensive use of Theorem~\ref{thm:equiv} and  classical results in convex analysis. Much is proved exclusively for an $\ell_2$ set inner product and the extent to which these results can be generalised  is not obvious.

In the final section, we show how the set inner product can be used to infer convex sets in comparative biology. Given an evolutionary tree and convex sets at the leaves (e.g. representing environmental niches) we use functional analysis techniques to infer niche geometries (convex sets) for ancestral species.

\section{Characterization of set inner products}

The {\em support function}  of a convex body $A \in \sK$ is defined 
\[h_A:\Re^n \rightarrow \Re:x \mapsto \sup\{a\cdot x:a \in A\}.\]
where $\cdot$ is 
the usual dot product in $\mathbb{R}^n$.
The properties of support functions are reviewed comprehensively in \cite{Schneider2014a}. We let $\sH$ denote the set of support functions for convex bodies in $\sK$. Then $h \in \sH$ if and only if  it is {\em sublinear}: for all $x,y \in \Re^n$ and $\alpha,\beta \geq 0$ we have
\begin{equation}
h(\alpha x + \beta y) \leq \alpha h(x) + \beta h(y). \label{eq:sublinear}
\end{equation}
Furthermore, for all $A,B \in \sK$ and $\lambda \geq 0$ we have
\begin{equation}
h_{A+B} = h_A + h_B \label{eq:supAdd},
\end{equation}
\begin{equation} 
h_{\lambda A} = \lambda h_A. \label{eq:supScale}
\end{equation}
and
\begin{equation}
A \subseteq B \mbox{ if and only if } h_A(x) \leq h_B(x) \mbox{ for all $x$.} \label{eq:supportSubset}
\end{equation}

Let $\ssS^{n-1}$ denote the unit sphere in $\Re^n$. Every function  $h \in \sH$ is determined by its restriction to $\ssS^{n-1}$; we denote this restriction by $\oh$. Let $\sH_S = \{\oh:h \in \sH\}$.  
The Hausdorff metric on $\sK$
\[d_H(A,B) = \min\{r: A \subseteq B + r\ssS^{n-1} \mbox{ and } B \subseteq A + r\ssS^{n-1}\}\]
has an elegant re-expression in terms of support functions
\[d_H(A,B) = \|\oh_A - \oh_B\|_\infty =  \sup\{|h_A(x) - h_B(x)| :  x \in \ssS^{n-1}\},\]
see \cite{Hormander1955a}. Hence the set of convex bodies on $\Re^n$ with the Hausdorff metric can be embedded naturally into a Banach space \cite{Hormander1955a,Radstrom1952a}. See \cite{Schmidt1986a} for useful generalizations of these results.  Schneider \cite[pg. 45] {Schneider2014a} observed that the span of $\sH_S$ is dense in $\ssC(\ssS^{n-1})$,  the algebra of continous functions on $S^{n-1}$ (with respect to the $L_\infty$ norm), and that every $f \in \ssC(\ssS^{n-1})$ can be approximated arbitrarily closely by the difference $\oh_A~-~\oh_B$ of two support functions. 

The Banach space embedding has proven a valuable tool for the study of random sets \cite{McClureETAL1975a,MolchanovETAL2005a} and also leads naturally to an $L_p$ metric for $\sK$:
\begin{equation}
d_p(A,B) =\| \oh_A  - \oh_B \|_p
\label{eq:L2distance}
\end{equation}
see \cite{Vitale1985a,McClureETAL1975a}.  Arnold and Wellerding proposed a Sobolev distance on $\sK$ \cite{ArnoldETAL1992a}. Several authors have used the fact that this, like the $L_2$ distance given by \eqref{eq:L2distance}, are inner product distances (e.g. \cite{Arnold1989a,Shephard1966a}). Here we attempt to put these results in a general setting.  

\begin{theorem} {\bf (Set inner product equivalence theorem)} \label{thm:equiv}
Let $\ddp$ be a bivariate map on $\sK$. Then $\ddp$ is a set inner product if and only if there is an inner product $\dd{\,\cdot , \cdot \,}_H$ on $\sp(\sH)$ such that 
\begin{equation}
\dd{A,B} = \dd{h_A,h_B}_H \label{eq:equiv}
\end{equation}
for all  $A,B \in \sK$.
\end{theorem}
\begin{proof}
If there is such an inner product on $\sp(\sH)$ then the set inner product axioms for $\ddp$ follow directly from the definitions and properties of inner products and support functions. 

For the converse, let $\ddp$ be a set inner product. Define $F:\sH \times \sH \rightarrow \Re$ by $F(h_A,h_B) = \dd{A,B}$ for all $A,B$. We will use $F$ to construct an inner product $G$ on $\sp(\sH)$ which agrees with $F$ on $\sH$. 
As $\sH$ is a convex cone, every element of $\sp(\sH)$ can be expressed as the difference of two elements in $\sH$. Define a function $G$ on $\sp(\sH) \times \sp(\sH)$ by
\[G(f-g,h-k) = F(f,h) - F(f,k) - F(g,h) + F(g,k).\]
To see that $G$ is well defined, suppose that $f_1,f_2,g_1,g_2,h_1,h_2,k_1,k_2 \in \sH$. If ${f_1-g_1 = f_2 - g_2}$ then 
\[F(f_1,h_1) + F(g_2,h_1) = F(f_1+g_2,h_1) = F(f_2+g_1,h_1) = F(f_2,h_1) + F(g_1,h_1)\]
and
\[F(f_1,k_1) + F(g_2,k_1) = F(f_1+g_2,k_1) = F(f_2+g_1,k_1) = F(f_2,k_1) + F(g_1,k_1)\]
from which we obtain
\[F(f_1,h_1) - F(g_1,h_1) - F(f_1,k_1) + F(g_1,k_1) = F(f_2,h_1) - F(g_2,h_1) - F(f_2,k_1) + F(g_2,k_1).\]
Similarly, if $h_1-k_1 = h_2 - k_2$ then
\[F(f_2,h_1) - F(g_2,h_1) - F(f_2,k_1) + F(g_2,k_1) = 
F(f_2,h_2) - F(g_2,h_2) - F(f_2,k_2) + F(g_2,k_2).\]
Hence if $f_1-g_1 = f_2 - g_2$ and $h_1-k_1 = h_2 - k_2$ then
\begin{equation*}
F(f_1,h_1) - F(g_1,h_1) - F(f_1,k_1) + F(g_1,k_1) = F(f_2,h_2) - F(g_2,h_2) - F(f_2,k_2) + F(g_2,k_2) \end{equation*}
We will show that $G$ is an inner product on $\sp(\sH)$.  

\begin{enumerate}[label=(\roman*)]
\item As $F$ is symmetric, so is $G$.
\item $G$ is additive, since
\begin{align*}
G((f_1-g_1)+(f_2-g_2),h-k) & = 	F(f_1+f_2,h) - F(f_1+f_2,k) - F(g_1+g_2,h) + F(g_1+g_2,k) \\
&=
F(f_1,h) - F(f_1,k) - F(g_1,h) + F(g_1,k) \\ & \quad \quad +
F(f_2,h) - F(f_2,k) - F(g_2,h) + F(g_2,k) \\
& = G(f_1-g_1,h-k)+G(f_2-g_2,h-k).
\end{align*}
\item For all $\lambda \geq 0$
\begin{align*}
G(\lambda(f-g),h-k) &= F(\lambda f,h) - F(\lambda f,k) - F(\lambda g,h) + F(\lambda g,k) \\ 
&= \lambda G(f-g,h-k)
\end{align*}
and 
\begin{align*}
G(-(f-g),h-k) & = G(g-f,h-k) \\
& = F(g,h) - F(g,k) - F(f,h) + F(f,k)\\
& = -G(f-g,h-k).	
\end{align*}
Hence $G$ is linear (and bilinear by symmetry).
\item
For all $f,g \in \sH$ there are $A,B$ such that $f = h_A$, $g = h_B$. Then 
\begin{align*}
G(f-g,f-g) &= F(f,f) - 2F(f,g) + F(g,g) \\
	& = \dd{A,A} - 2\dd{A,B} + \dd{B,B} \\ 
	& \geq \dd{A,A} - 2\sqrt{\dd{A,A}\dd{B,B}} + \dd{B,B} & \mbox{ by (A4)}\\
	& = (\sqrt{\dd{A,A}} - \sqrt{\dd{B,B}})^2 \\
	& \geq 0.
\end{align*}
\end{enumerate}
To see that $G$ is positive definite notice 
\[0 = G(f-g,f-g) \implies \dd{A,A}=\dd{B,B}=\dd{A,B}. \] 
Axiom (A4) implies $A = \lambda B$ for some non-negative scalar $\lambda$ and (A2) implies $\lambda = \lambda^2 = 1$ and hence $\lambda =1$.  Thus $f-g=0$. 
We have then that the map $G$ 
is an inner product on $\sp(\sH)$. Furthermore, for any $A,B \in \sK$ we have $\dd{A,B} = F(h_A,h_B) = G(h_A,h_B)$, giving \eqref{eq:equiv}. 
\qqed \end{proof}

\begin{corollary}
Let $\ddp$ be a bivariate map on $\sK$. Then $\ddp$ is a set inner product if and only if there is an inner product $\dd{\,\cdot , \cdot \,}_S$ on $\sp(\sH_S)$ such that 
\[\dd{A,B} = \dd{\oh_A,\oh_B}_S\]
for all  $A,B \in \sK$.
\end{corollary}
\begin{proof}
If  $\dd{A,B} = \dd{\oh_A,\oh_B}_S$ then axioms (A1)---(A4) follow directly from the definitions and properties of inner products and support functions. Conversely, 
if $\ddp$ is a set inner product then by the theorem there is an inner product $\ddp_H$ on $\sp(\sH)$ such that $\dd{A,B} = \dd{h_A,h_B}_H$ for all $A,B$. For each $\overline{f} \in  \sp(\sH_S)$  define $f \in \sp(\sH)$ by $f(x) = \|x\| \overline{f}\left( \frac{x}{\|x\|} \right)$ for all $x \neq 0$. Defining 
\[\dd{\overline{f},\overline{g}}_S = \dd{f,g}_H\]
gives the required inner product on $\sp(\sH_S)$.
\qqed \end{proof}

\begin{example} \label{def:standardInner}
Let $\phi(x)$ denote the standard Gaussian density on $\Re^n$ and define  
\[\dpt{A,B} = \int h_A(x) h_B(x) d\phi(x).\]
Then by Theorem~\ref{thm:equiv} $\dpt{\,\cdot  , \cdot  \, }$ is a set inner product. Note that because $h_A(\lambda x) = \lambda h_A(x)$ for all $\lambda \geq 0$ we have that 
\[\int h_A(x) h_B(x) d\phi(x) = \int_{\ssS^{n-1}} \oh_A(x) \oh_B(x) d\mu(x)\]
for the Haar measure $\mu$ on $\ssS^{n-1}$, the unit sphere in $\Re^n$. 
\end{example}

Suppose that $\ddp$ is an inner product on $\Re^n$ and that we define the support functions with respect to this inner product  in place of the dot product. If we pick an arbitrary inner product $\ddp_H$ on $\sp(\sH)$ then there is no guarantee that the resulting {\em set} inner product 
\[ \dd{A,B}_s = \dd{h_A,h_B}_H \]
will agree with the original inner product on singletons. That is, we could well have 
\[ \dd{\{a\},\{b\}}_s =  \dd{h_{\{a\}},h_{\{b\}}}_H \neq \dd{a,b}\]
for some $a,b \in \Re^n$.  However for each such inner product $\ddp_H$ we can nevertheless find an invertible linear map $f:\Re^n \rightarrow \Re^n$ such that 
 \[ \dd{h_{\{a\}} \circ f ,h_{\{b\}} \circ f}_H = \dd{a,b}.\]

\begin{example} \label{ex:realLine}
The support functions for closed intervals on the real line are determined by two values, $h_{[a,b]}(1) = b$ and $h_{[a,b]}(-1) = -a$. Hence a generic set inner product on the real line can be written
\[\dd{[a,b],[c,d]} = \mat{b & -a} M \mat{d\\-c}\]
where $M$ is positive definite. The constraint that 
\[ \dd{[a,a],[c,c]} = ac\]
is obtained by scaling $M$ so that $\mat{1&-1} M \mat{1\\-1} = 1$. 
\end{example}

We made explicit use of the Cauchy-Schwartz inequality (A4) to prove that every set inner product corresponds to an inner product on the support functions. It is natural to ask whether this was strictly necessary, given that the inequality is implied by the axioms for standard inner products. The next construction gives a function which satisfies all of the  set inner product axioms except the Cauchy-Schwartz inequality.

\begin{proposition}
There is a bilinear function on $\sK$ which satisfies (A1), (A2) and (A3) but not  the Cauchy-Schwartz inequality (A4).
\end{proposition}
\begin{proof} 
Let $\{e_1,e_2\}$ be the standard basis of $\Re^2$. Define $F:\sK \times \sK \to \Re$ 
where
 \begin{align*}F(A,B) = & \frac{1}{8}\left( h_A (e_1) h_B (e_1) +  h_A (e_2) h_B (e_2) +  h_A (-e_1) h_B (-e_1) +  h_A (-e_2) h_B (-e_2) \right)\\
 &+ \left[ \left(  h_A (e_1)+ h_A (e_2) -  h_A (e_1+e_2) \right)\left(  h_B (-e_1)+ h_B (-e_2) - h_B (-e_1-e_2) \right) \right]\\
 &+ \left[ \left(  h_B (e_1)+ h_B (e_2) -  h_B (e_1+e_2) \right)\left(  h_A (-e_1)+ h_A (-e_2) - h_A (-e_1-e_2) )\right) \right].\\
 \end{align*} 
Then $F$ is symmetric and bilinear. For any $A \in \sK$  we have
 \begin{align*}&F(A,A) =  \frac{1}{8}\left(h_A(e_1)^2 + h_A(e_2)^2 + h_A(-e_1)^2 + h_A(-e_2)^2\right) \notag \\
 &+2 \left[ \left( h_A(e_1)+h_A(e_2) - h_A(e_1+e_2) \right)\left(h_A(-e_1)+ h_A(-e_2)-h_A(-e_1-e_2) \right) \right]\\
 & \geq 0.
 \end{align*}
 Now suppose $F(A,A) = 0$.  Then since 
 \begin{align*}
 h_A(e_1)+h_A(e_2) - h_A(e_1+e_2)&\geq 0  \intertext{  and  } h_A(-e_1)+ h_A(-e_2)-h_A(-e_1-e_2)  &\geq 0
  \end{align*}
  we have
 \[h_A(e_1)=h_A(e_2)=h_A(-e_1)=h_A(-e_2)=0.\] 
 Hence $h_A(\lambda e_i) = h_A(-\lambda e_i) = 0$ for $i=1,2$ and all $\lambda \geq 0$ and so for all $x \in \Re^2$,
\[h_A(x)= h_A(x_1e_1+x_2e_2) \leq h_A(x_1 e_1) + h_A(x_2 e_2) \leq 0.\]
By \eqref{eq:supportSubset}, $A \subseteq \{0\}$, giving equality as $A$ is non-empty.

We have shown that $F$ satisfies the first four axioms of a set inner product. We now show that it fails the last axiom. Let
$A$ be the unit circle in   $\mathbb{R}^2$ and $B$ be the triangle through the points $(0,0), (1,0), (0,1)$. Then
\[h_A(e_1) = h_A(e_2) = h_A(-e_1) = h_A(-e_2) = 1;\]
\[h_A(e_1+e_2) = h_A(-e_1-e_2) = \sqrt{2};\]
\[h_B(e_1) = h_B(e_2) = h_B(e_1+e_2) = 1;\]
\[h_B(-e_1) = h_B(-e_2) = h_B(-e_1-e_2) = 0.\]
%F(A,A) = 1/8 x (4) + 2(2-\sqrt{2})^2 
%F(A,B) = 1/8 x (2) + (2 - \sqrt{2}) x 0  +  1 x (2 - \sqrt(2)) 
%F(B,B) = 1/8 x 2  +  2 x 1 x 0
%FAA := 4/8 + 2*(2-sqrt(2))^2;FAB := 1/8 * 2 + (2 - sqrt(2)); FBB := 2/8;
We then have
\begin{align*}
2F(A,B) - 2\sqrt{F(A,A)F(B,B)} & = (\sqrt{F(A,A)} - \sqrt{F(B,B)})^2 - F(A,A) + 2 F(A,B) - F(B,B) \\
& \geq -F(A,A) + 2 F(A,B) - F(B,B) \\
& = 6 \sqrt{2} - \frac{33}{4} \\
& \approx 0.235
\end{align*}
so that the Cauchy-Schwarz inequality fails.
\qqed \end{proof}

\section{Lines and subspaces in $\sK$}

Having established that set inner products are essentially standard inner products (for support functions) we explore  implications for a geometry of the space of convex bodies. 

\begin{definition}
A subset $\sA \subseteq \sK$ is an {\bf affine subspace} if $\{h_A : A \in \sA \}$ equals the intersection of $\sH$ and an affine subspace of  $\sp(\sH)$. A {\bf subspace} of $\sK$ is an affine subspace which includes the zero-set $\{0\}$. \end{definition}

\begin{example}
Figure~\ref{fig:circleSquare} depicts elements on the line segment 
\[%\{\alpha S + (1-\alpha) C: 0 \leq \alpha \leq 1\} = 
\{A \in \sK_2 : h_A = \alpha h_S + (1-\alpha) h_C \mbox{ for some $\alpha \in [0,1]$} \}\]
connecting a square $S$ and a circle $C$ in $\sK_2$. Figure~\ref{fig:2dplane} depicts the two dimensional subspace of $\sK_2$ spanned by a circle $C$ and a triangle $T$,
\[\{A \in \sK_2: h_A = \alpha h_C + \beta h_T \mbox{ for some $\alpha,\beta \geq 0$}\}.\]
Note that the first example is bounded and the second example is unbounded.  Any subspace which contains a non-zero convex body $A$ also contains $\lambda A$ for all $\lambda \geq 0$, so that the subspace is unbounded. The same need not hold for affine subspaces.
\end{example}

\begin{figure}[htp]
\begin{center}
\includegraphics[width=0.8\textwidth]{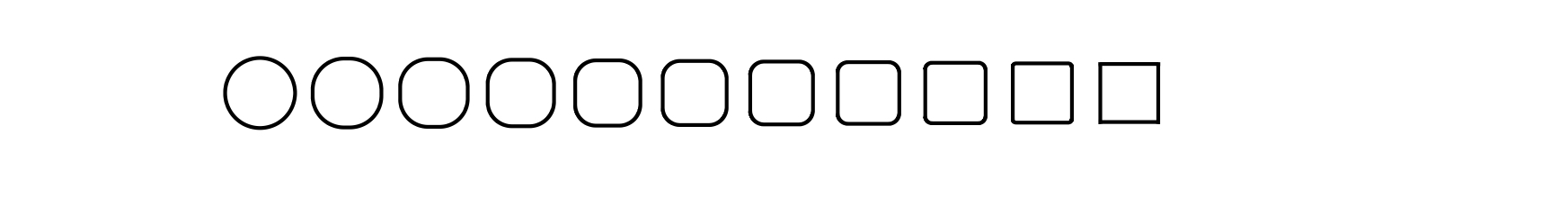}
\caption{\label{fig:circleSquare} The line segment in $\sK_2$ connecting a circle to a square.}
\end{center}
\end{figure}

\begin{figure}[htp]
\begin{center}
\includegraphics[width=0.8\textwidth]{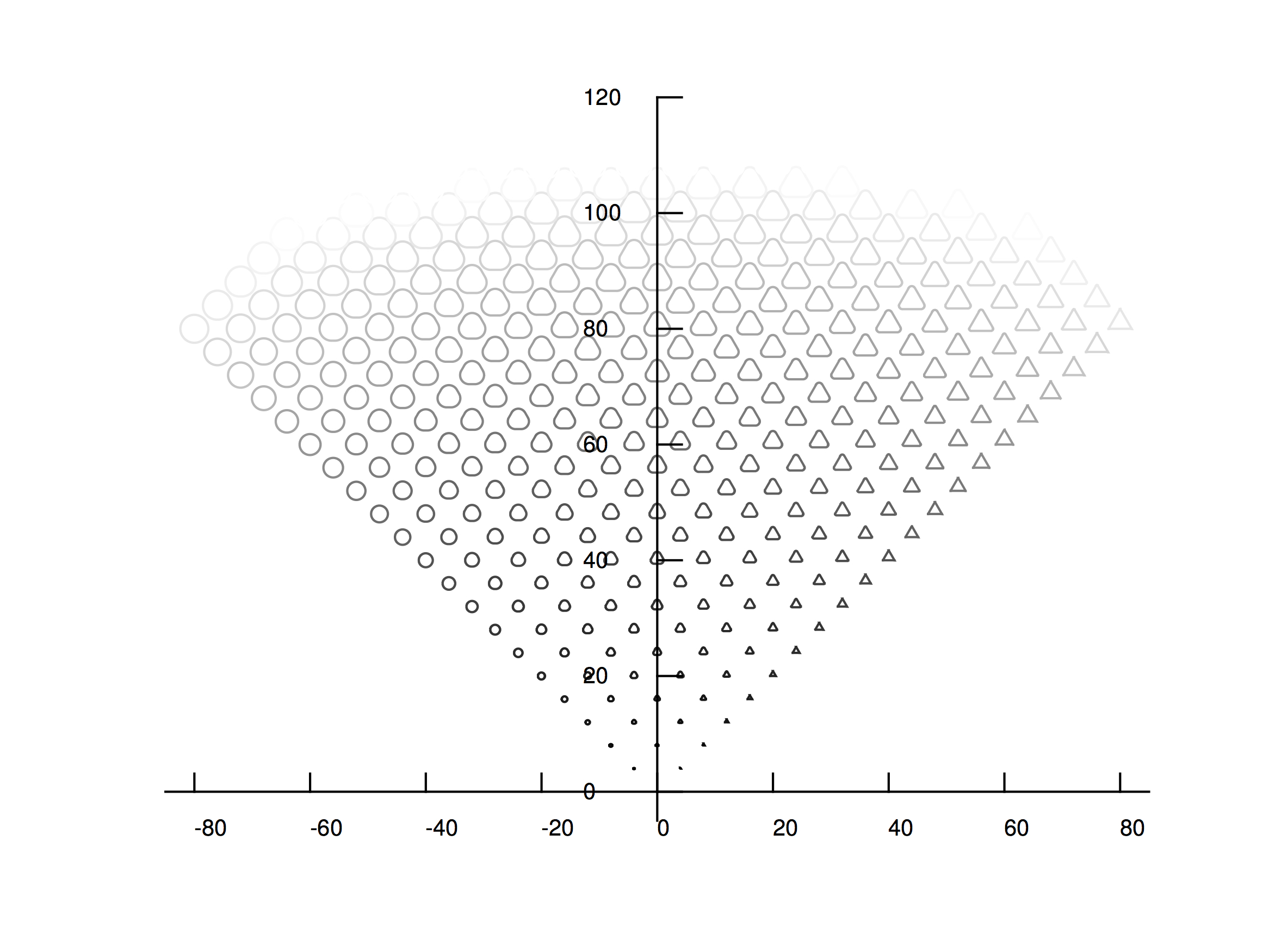}
\caption{\label{fig:2dplane} The two dimensional subspace of $\sK_2$ spanned by a circle at $(-40,40)$ and an equilateral triangle at $(40,40)$. }
\end{center}
\end{figure}

We say that a subset $\ell$ of $\sK$ is a {\bf line} if the corresponding set of support functions $\{h_A: A \in \ell\}$ equals the intersection of a line with $\sH$.  Because of the intersection with $\sH$, lines can have endpoints and subspaces can be bounded subsets. 

\begin{proposition}
Every line $\ell \subseteq \sK$ is one of three types (where $A,B \in \sK$,  $x \in \Re^n$):
\begin{enumerate}
\item {\em Translations:} $\{A + t x : t \in \Re\}$
\item {\em Rays:} $\{t A + B : t \geq 0\}$
\item {\em Segments:} $\{ t A + (1-t)B: 0 \leq t \leq 1\}$.
\end{enumerate}
\end{proposition}
\begin{proof}
Suppose that $A,B \in \ell$, so that $X \in \ell$ if and only if 
\[h_X \in \{ h_B + t(h_A - h_B)  : t \in \Re\}.\]
We consider three cases. 

Case 1:  $h_A - h_B = h_{\{x\}}$ for some $x \in \Re^n$.  Then $\ell$ is a translation. \\
 
Case 2:   $h_A - h_B = h_C$ for some non-singleton $C \in \sK$ so $\ell = \{ h_B + t \, h_C: t \in \Re\}$.  As $C$ is not a singleton, $h_C$ is not linear and there are $x,y$ such that $h_C(x+y) < h_C(x) + h_C(y)$. We then have that if $t$ is sufficiently negative,
\[ h_B(x+y) + t h_C(x+y) > h_B(x) + t h_C(x) + h_B(y) + t h_C(y)\]
so that  the set $\{t : h_B + t\, h_C \in \sH\}$ is bounded below. 

Let $\lambda  = \min \{t: h_B + t\, h_C \in \sH\}$. Then $X \in \ell$ if and only if
  \[h_X \in  \{ (h_B +\lambda h_C) + t \,h_C : t \geq 0 \}\]
  so that $\ell$ is a ray.\\
  
  Case 3: $h_A - h_B \not \in \sH$. Let $f = h_A - h_B$. As $f \not \in \sH$ there is $x,y \in \Re^n$ such that $f(x) + f(y) - f(x+y) < 0$. Hence if 
  \[t > \frac{h_B(x) + h_B(y) - h_B(x+y)}{f(x+y) - f(x) - f(y)}\]
   then 
$h_B + t(h_A - t_B)  \not \in \sH$.   Thus the line has an endpoint at $A$ or at some point on the other side of $A$ from $B$. The same holds in the opposite direction.
\qqed \end{proof}

The characterisation of endpoints was considered (in three dimensions) by \cite{Vincensini1936a} and is closely related to Minkowski summands and decompositions. For $K,L \in \sK$ we say that $L$ is a {\bf summand} of $K$ if there is $M \in \sK$ such that $L+M = K$.  Schneider \cite[pg. 157]{Schneider2014a} shows that this is equivalent to being able to cover $K$ with translates of $L$ contained in $K$.

\begin{proposition}
Suppose $A,B \in \sK$. Then $A$ is an endpoint on the line through $A$ and $B$ if and only if  for all $\epsilon \in (0,1)$, the convex body $\epsilon B $ is not a summand of $A$. 
\end{proposition}
\begin{proof}
Suppose that there is $C$ such that $\epsilon B + C = A$. Let $t = \frac{-\epsilon}{1-\epsilon}$. Then $t<0$ and 
\[ (1-t) h_A + t h_B = \frac{1}{1-\epsilon} h_C\]
so $\frac{1}{1-\epsilon} C$ is on the line through $A$ and $B$ on the opposite side of $A$ from $B$.

Conversely, if there is $h_C \in \sH$ and $t < 0$ such that $h_C = (1-t) h_A + t h_B$ then 
\[\frac{-t}{1-t} h_B + \frac{1}{1-t} h_C =  \frac{-t}{1-t} h_B + \frac{1-t}{1-t} h_A + \frac{t}{1-t} h_B = h_A\]
so 
\[\left(\sfrac{-t}{1-t}\right) B + \left( \sfrac{1}{1-t} \right) C = A\]
 $\left(\sfrac{-t}{1-t} \right) B$ is  a summand of $C$.    
 \qqed \end{proof}

Which convex bodies can be endpoints? Intuitively, these are convex bodies on the boundary of $\sH$. 

\begin{proposition}
Every convex body is the endpoint of some line. 
\end{proposition}
\begin{proof}
Suppose that $A \in \sK$ and let $x$ be an exposed point of $A$ (that is, there is a supporting hyperplane $H$ of $A$ such that $\{x\} = H \cap A$). 
%NOTE: the existence of exposed points follows indirectly from Straszewicz's theorem  - but I think it is "obvious". 
Let $B$ be any line segment  containing $x$ and contained in $H$. Then any translation of a scaled version of $B$ which contains $x$ will also contain other points in $H$ which are therefore not in $A$. It follows that there is  $\epsilon B$ does not slide freely around $A$ for any $\epsilon \in (0,1)$, and $A$ is an endpoint of the line through $A$ and $B$.
\qqed \end{proof}

A convex body which has no summands (other than scaled and translated versions of itself) is said to be {\bf indecomposable}.  These convex bodies are exactly those which are endpoints of {\em every} line they are contained in. In $\Re^2$ the indecomposable bodies are the line segments and triangles \cite{Silverman1973a}, while for higher dimensions, the indecomposable bodies form a dense subset of $\sK_n$  with respect to the Hausdorff metric (see \cite[pg.165-6]{Schneider2014a} for details).

\section{Geometry of the set inner product}

Let $\ddp$ be a set inner product on $\sK_n$, where we write 
\[\dd{a,b} = \dd{\{a\},\{b\}}\]
for the induced inner product. 

For any convex body $A \in \sK$, the map $x \mapsto \dd{A,\{x\}}$ is a continuous linear functional. By the Riesz-Fr\'echet theorem there is a unique element $k_A \in \Re^n$  such that
\[ \dd{A,\{x\}} = \dd{k_A,x}\]
for all $x \in \Re^n$. We say that $k_A$ is the {\bf center} of $A$.

\begin{example}
Recall  the $L_2$ set inner product
\[\dpt{A,B} = \int h_A(x) h_B(x) d\phi(x)\]
defined in Example~\ref{def:standardInner}. The center of a set for this set inner product is given by the {\bf Steiner point}
\[s_A = \int z \,\, h_A(z) d\phi(z)\]
since for all $x \in \Re^n$ we have $h_{\{x\}}(z) = \dd{x,z}$ and 
\[\dpt{A,\{x\}} = \int  \dd{z,x} h_A(z) d \phi(z) = \dd{\int z\,\, h_A(z)  \,d \phi(z),x }  = \dd{s_A,x}.\]
\end{example}

The Steiner point for a convex body always lies in the relative interior of the set 
\cite[pg. 43]{Schneider2014a} so for this set inner product, $k_A \in A$. Furthermore, for any set inner product we have for all $x \in \Re^n$ that 
\[\dd{k_{\{a\}} , x} = \dd{\{a\},\{x\}} = \dd{a,x}\]
so $k_{\{a\}} = a$. However in general, we do not have that $k_A \in A$.

\begin{example}
Let $\ddp$ be a set inner product for closed intervals on the real line, where
\[ \dd{[a,b],[c,d] } = \mat{b&-a} M \mat{d \\ -c}\]
and 
\[ \dd{[a,a],[c,c] } = ac.\]
Then for $A = [a,b]$ we have $k_A =  \mat{b&-a} M \mat{1 \\ -1}$ so that for all $x \in \Re$,
\[ k_A x =  \mat{b&-a} M \mat{x \\ -x} = \dd{[a,b],[x,x]}.\]
As a concrete example, let
\[M = \mat{2 & & 3 \\ 3 & & 5}\]
so $M$ is positive definite and $\mat{1&-1} M \mat{1 \\ -1} = 1$. However 
\[k_A =  \mat{b&-a} \mat{2 & &3 \\ 3&  & 5} \mat{1 \\ -1} = 2a - b.\]
Hence if $A = [0,1]$ then $k_A \not \in A$.
\end{example}

We say that a convex body $A \in \sK$ is {\bf centered} if $k_{A} = 0$, which holds if and only if $\dd{A,\{x\}} = 0$ for all $x$.  We write $A_0 = A - \{k_A\}:=A + \{-k_A\}$, noting that for all $x$, 
\[\dd{A_0,\{x\}} = \dd{A,\{x\}} + \dd{ -k_A , \{x\}} = 0\]
so that $A_0$ is centered.

\begin{proposition} Suppose that $A_0 = A - k_A$ and $B_0 = B - k_B$. Then
\[\dd{A,B} = \dd{A_0,B_0} + \dd{k_A,k_B}.\]
\end{proposition}
\begin{proof}
We have
\begin{align*}
\dd{A,B} & = \dd{A_0 + k_A,B_0 + k_B}  \\
& = \dd{A_0,B_0} + \dd{A_0,k_B} + \dd{k_A,B_0} + \dd{k_A,k_B} \\
& = \dd{A_0,B_0} + \dd{A,k_B} - \dd{k_A,k_B} + \dd{k_A,B} - \dd{k_A,k_B} + \dd{k_A,k_B}\\
& = \dd{A_0,B_0} + \dd{k_A,k_B}.
\end{align*}
\item For all $x \in \Re^n$ we have 
\[ \dd{a,x} = \dd{\{a\},\{x\}} = \dd{k_a,x}.\]
Thus $k_a - a = 0$.
\qqed \end{proof}

For the remainder of this section we explore properties of the $L_2$ set inner product $\dpt{\cdot,\cdot}$ introduced in Example~\ref{def:standardInner}. It is not clear which of these results can be generalised.

\begin{proposition} \label{propL2}
Suppose that $\ddp$ equals the $L_2$ inner product
\[\dpt{A,B} = \int h_A(x) h_B(x) d\phi(x).\]
\begin{enumerate} 
\item For all $A \in \sK$, $k_A \in A$.
\item If $A$ and $B$ are centered  then 
$ \dd{A,B} = 0$ if and only if $A = \{0\}$ or $B = \{0\}$.
\item If $0 \in A \subseteq B$ then $\dd{A,A} \leq  \dd{B,B}$.
\end{enumerate} 
\end{proposition}
\begin{proof}
\begin{enumerate}
\item See \cite[pg. 43]{Schneider2014a}.
\item 
Suppose by way of contradiction that neither $A$ nor $B$ equal $\{0\}$.  As $0 \in A \cap B$ we have that $h_A$ and $h_B$ are both non-negative. We will show 
that there is $x$ such that $h_A(x) h_B(x) > 0$, implying that 
\[ \int h_A(x) h_B(x) d\phi(x) > 0.\]

If there is $x$ such that the strict half-space $\{y: \dd{x,y} > 0 \}$ intersects both $A$ and $B$ then 
\[h_A(x) h_B(x) \geq \dd{a,x}  \dd{b,x} > 0.\]
Otherwise, $A$ and $B$ are both contained in a line $\{ tv:v \in \Re \}$ and they only intersect at $0$. Without loss of generality, 
$A \subset \{ tv : t \geq 0\}$. However then $h_A(x) h_{\{v\}}(x) \geq 0$ for all $x$ and $h_A(v) h_{\{v\}}(v)  > 0$, implying
\[ \dd{A,\{v\}} =  \int h_A(x) h_{\{v\}}(x) d\phi(x) > 0 \]
and $k_A \neq 0$.
\item As $0 \in A$ and  $A \subseteq B$ we have $0 \leq h_A(x) \leq h_B(x)$ for all $x$. 
\end{enumerate} 

\qqed \end{proof}

We write $A - A = A + (-1)A = \{ a_1-a_2:a_1,a_2 \in A\}$.

\begin{lemma}
If $A \cap B \neq \emptyset$ then 
\[(A \cup B) - (A \cup B) \subseteq (A-A) + (B-B).\]
\end{lemma}
\begin{proof}
Since $0 \in A-A$ and $0 \in B-B$ we have $(A-A) \cup (B-B) \subseteq (A-A) + (B-B)$. Thus if $x,y \in A$ ot $x,y \in B$ then $x-y \in (A-A) + (B-B)$. 

Suppose $z \in A \cap B$. If $x \in A$ and $y \in B$ then $x-z \in A-A$ and $z - y \in B-B$ so $x-y = (x-z) + (z-y) \in (A-A) + (B-B)$.
\qqed \end{proof}

A {\bf diversity} $(X,\delta)$ is an extension of the concept of a metric space to include comparisons of more than two points at a time \cite{BryantTupper2012}. Here $X$ is a set and $\delta$ is a function on finite subsets of $X$ satisfying
\begin{enumerate}
\item[(D1)] $\delta(A) \geq 0$ and $\delta(A) = 0 \Leftrightarrow |A| \leq 1$
\item[(D2)] For all $A,B,C$ with $B\neq \emptyset$, $\delta(A \cup B) + \delta(B \cup C) \geq \delta(A \cup C)$.
\end{enumerate}
We note that (D1) and (D2) imply that $\delta$ is monotonic with respect to inclusion: if $A \subseteq B$ then $\delta(A) \leq \delta(B)$.

\begin{proposition}
For any finite $A \subseteq X$ define
\[ \delta(A) = \sqrt{ \dpt{ \conv(A-A) , \conv(A-A) }}.\]
Then $\delta$ is a diversity.
\end{proposition}
\begin{proof}
We check the axioms. If $A = \{x\}$ then $A-A = \{0\}$ and $\delta(A) = 0$. \\
If $A \subseteq B$ then $A-A \subseteq B-B$ and $\conv(A-A) \subseteq \conv(B-B)$. As $0 \in A-A$ we have $\delta(A) \leq \delta(B)$ by Proposition~\ref{propL2} part 3. That is $\delta$ is monotonic. \\
If $a,b \in A$ and $ a \neq b$ then $\conv(A-A) \neq \{0\}$ and $\delta(A) \geq \delta(\{a,b\}) > 0$. 
Suppose that $z \in A \cap B$. By the lemma we have
\[(A \cup B) - (A \cup B) \subseteq (A-A) + (B-B)\]
which implies
\[\conv( (A \cup B) - (A \cup B)) \subseteq \conv(A-A) + \conv(B-B).\]
Hence
\begin{align*}
\delta(A \cup B) &=  \dpt{ \conv(A \cup B - A \cup B) ,\conv(A \cup B - A \cup B) }^\half \\
& \leq \dpt{ \conv(A -A)  + \conv(B-B) ,\conv(A -A)  + \conv(B-B))  }^\half\\
& \leq  \dpt{ \conv(A -A) , \conv(A -A) }^\half +  \dpt{ \conv(B-B),\conv(B-B)}^\half\\
& = \delta(A) + \delta(B).
\end{align*}
This, and the fact that $\delta$ is monotonic, implies (D2).
\qqed \end{proof}

\section{Projecting convex bodies onto phylogenies}

In this final section we present an application of set inner products to computational biology. In ecology, a {\bf niche} can be defined as the set of environments in which an organism or species can survive/prosper. One can model the niche of a homogeneous population as a convex subset within some abstract space of environmental parameters. We are interested in studying how niches change over time.

A {\bf phylogeny} is a tree $T$ with leaves (degree one nodes) labelled bijectively by species and internal nodes corresponding to ancestral species. We assume that the tree is {\bf binary}: every internal node has degree three. The phylogeny is  model of the history of speciations or splits given rise to the species in the sample.

The general problem we consider is the following: suppose we have inferred  niche information for each of the observed (leaf) species;  can we infer how those niche differences evolved over time, and can we infer anything about the niches occupied by ancestral species? 

A comprehensive answer to these questions will require models of random and directed change based on stochastic processes together with assessment of model and statistical uncertainty. Much of the original motivation for this paper was setting up the machinery for such models. However we will see that set inner products already provide a direct estimate based on a  Steiner-tree problem for ancestral niches.

We first consider a special case. Consider $N$ species all children of a single ancestral node. This is often called a {\bf star tree} (Fig.~\ref{fig:star}). Let $\ddp$ be an inner product. We measure the distance between convex sets using the associated norm:
\[d(A,B) =  (\dd{h_A-h_B,h_A-h_B})^{1/2}.\]
Each niche is modelled by a convex body. The aim is to determine a convex body $X$ for the ancestral node which minimizes the total {\em squared parsimony} length of the tree:
\[ \sum_{i=1}^N d(A_i,X)^2.\]

\begin{figure}[htp]
\begin{center}
\includegraphics[width=0.3\textwidth]{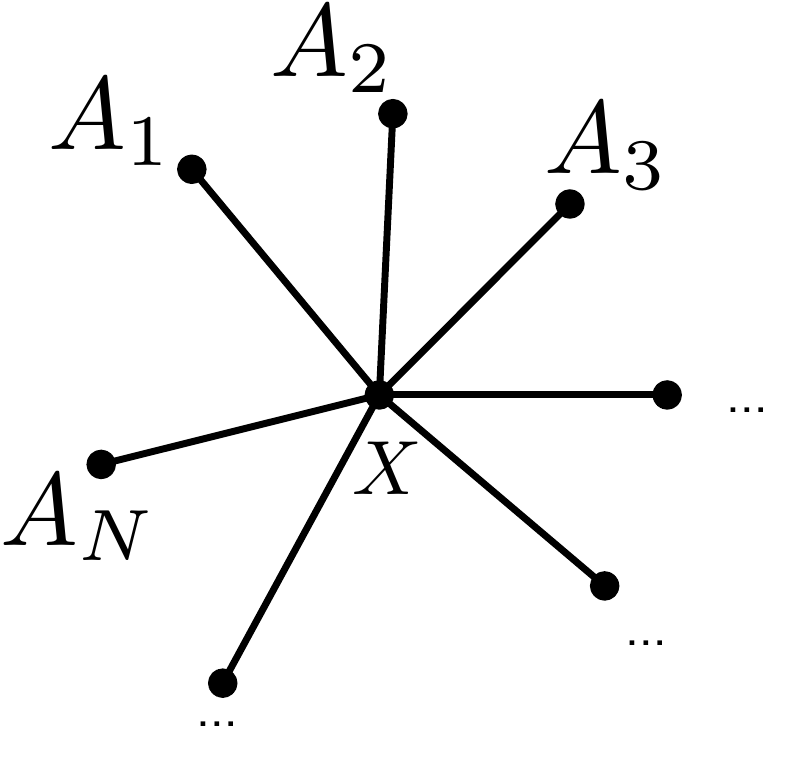}
\caption{\label{fig:star} A star tree}
\end{center}
\end{figure}

This problem can be solved for a general inner product space.

\begin{proposition} \label{prop:star}
The convex body $X$ which minimizes  $\sum_{i=1}^N d(A_i,X)^2$ is given by
\[X = \frac{1}{N} \sum_{i=1}^N A_i.\]
\end{proposition}
\begin{proof}
This is a direct consequence of two observations. First, that $g = \frac{1}{N} \sum_{i=1}^N h_{A_i}$ is the unique minimizer of 
\[ \min_f  \sum_{i=1}^N \| g - h_{A_i}\|^2 = \min_f  \sum_{i=1}^N \dd{g - h_{A_i},  g - h_{A_i}}\]
on the completion of $(\sp(\sH),\ddp)$. Second, that $g = \frac{1}{N} \sum_{i=1}^N h_{A_i}$ is in the convex hull of a set of support functions, so is itself a support function for a convex body $X$. 
\qqed \end{proof}

Let $T$ be a phylogeny with $N$ leaves and let $V(T)$, $L(T)$, and $E(T)$ denote the sets of nodes, leaves and edges of $T$. Let $\chi:L(T) \rightarrow \sK_n$ be a map from the leaves of $T$ to the set of convex bodies. An {\bf extension} of $\chi$ is a function $\hchi:V(T) \rightarrow \sK_n$ such that $\hchi(x) = \chi(x)$ for leaves $x \in L(T)$. The (squared parsimony) length of $\hchi$ is defined
\begin{equation} \ell(\hchi) = \sum_{\{x,y\}: E(T)} d(\hchi(x),\hchi(y))^2. \label{eq:length} \end{equation}
The aim is to find an extension $\hchi$ of minimum length.

\begin{theorem} \label{thm:shapeTree}
There exists a unique extension $\hchi$ of minimum length. For each $v \in V(T)$ the convex body  $\hchi(v)$ can be written as a non-negative linear combination  
\[\sum_{x \in L(T)} \lambda_{vx} \chi(x)\]
of convex bodies at the leaves (and with $\lambda_{xx} = 1$ for all $x \in L(T)$). Furthermore, the coefficients $\lambda_{vx}$ do not depend on the underlying set inner product, nor on the map $\chi$.
\end{theorem}

\begin{proof}
Suppose that $\hchi$ is an extension with minimum length. Let $v$ be any node in $V(T) \setminus L(T)$ and let $a,b,c$ be its neighbours. we have from Proposition~\ref{prop:star} that 
\begin{equation} \hchi(v) = \frac{1}{3} (\hchi(a) + \hchi(b) + \hchi(c)).\label{Laplacian} \end{equation}
That is $\hchi$ satisfies a type of Laplacian system. We will show by construction that these equations can be simultaneously satisfied for all internal nodes.  

Fix a leaf $x \in L(T)$ and direct all edges in $T$ away from $x$. For any node $v \neq x$ we let $p_x(v)$ denote the neighbor of $v$ on the path from $v$ to $x$. For each node $v$ define $\alpha_v = 0$ if $v$ is a leaf and otherwise let $\alpha_v = (3 - \alpha_a - \alpha_b)^{-1}$, where $a$ and $b$ are the two nodes with $p_x(a) = p_x(b) = v$.  For each internal node $v$ let $\lambda_{vx}$ be the product of $\alpha_u$ over all internal nodes $u$ on the path from $v$ to $x$.  We repeat this process for all $x$.

First note that if $v$ is any internal node, $a,b,c$ are its neighbors and $c = p_x(v)$ then 
\begin{align*}
\frac{1}{3} (\lambda_{ax} + \lambda_{bx} + \lambda_{cx}) & = \frac{1}{3} (\alpha_a \alpha_v +  \alpha_b \alpha_v  + 1)  \lambda_{cx} \\ 
& = \frac{1}{3} \left( \frac{\alpha_a}{3 - \alpha_a - \alpha_b} + \frac{\alpha_b}{3 - \alpha_a - \alpha_b}  + 1 \right) \lambda_{cx} \\
& = \alpha_v \lambda_{cx} \\
& = \lambda_{vx}.
\end{align*}
As this holds for all $x$ we have 
\[\hchi(v) = \frac{1}{3} (\hchi(a) + \hchi(b) + \hchi(c)).\]

To prove uniqueness, suppose that $\chi(y) = \{0\}$ for all $y$.  Fix $x$ and direct the edges in the tree as above. We claim that if $\hchi$ satisfies \eqref{Laplacian} then $\hchi(v) = \{0\}$ for all $v$. To this end, we show first that $\hchi(v) = \alpha_v \hchi(p(v))$ for all nodes $v \neq x$. The result holds trivially for all leaves $x$. Suppose that $v$ has neighbors $a,b,c$ and that $p(v) = c$ and that the $\hchi(a) = \alpha_a \hchi(v)$ and $\hchi(b) = \alpha_b \hchi(v)$. We have  
\[\hchi(v) = \frac{1}{3} (\hchi(a) + \hchi(b) + \hchi(c)) = \frac{1}{3} (\alpha_a \hchi(v) + \alpha_b \hchi(v) + \hchi(c)) \]
so
\[\hchi(v) = \frac{1}{3 - \alpha_a - \alpha_b} \hchi(c) = \alpha_v \hchi(c).\]
If $v$ is adjacent to $x$ then $\hchi(v) = \alpha_v 0$, from which we have $\hchi(v) = 0$ for all $v$. 
\qqed \end{proof}

The proof of Theorem~\ref{thm:shapeTree} leads directly to a time optimal $O(n^2)$ time algorithm for computing the coefficients $\lambda_{vx}$. 
Figure~\ref{fig:shapeTree} shows the result of applying the algorithm to a set of  simple 2D convex bodies on a tree with seven leaves. The original shapes  label the leaves of the tree while the shapes giving a minimal length label the nodes in the interior of the tree.  Source code for determining the coefficients $\lambda_{vx}$ is available from the corresponding author.

\begin{figure}
\begin{center}
\includegraphics[width=0.7\textwidth]{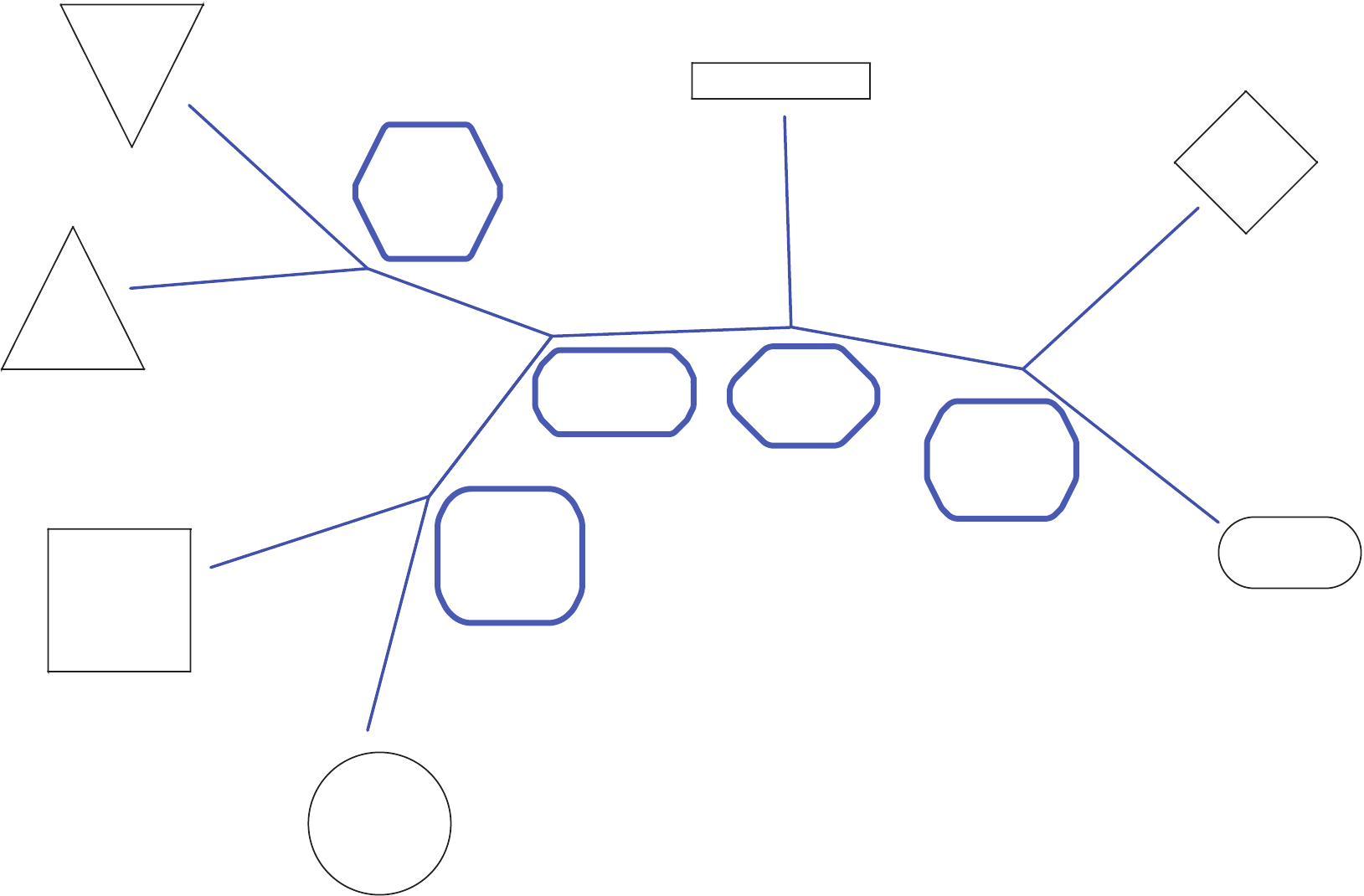}
\caption{\label{fig:shapeTree} A tree with seven leaves, each labelled by a convex shape (in black). The shapes minimizing the total length of the tree (Eq. \ref{eq:length}) are drawn in bold next to the corresponding internal nodes.}

\end{center}
\end{figure}

%\begin{acknowledgements}
%\end{acknowledgements}

\bibliographystyle{spmpsci}      % mathematics and physical sciences
\bibliography{SetInnerBib}   % name your BibTeX data base

%\begin{thebibliography}{}
%%
%% and use \bibitem to create references. Consult the Instructions
%% for authors for reference list style.
%\end{thebibliography}

\end{document}